\documentclass[11pt]{amsart}

\usepackage[letterpaper,margin=1in]{geometry}
\usepackage{amsmath,amssymb,amsthm,mathtools}
\usepackage{booktabs,longtable,array}
\usepackage{graphicx}
\usepackage{microtype}
\usepackage{hyperref}
\hypersetup{colorlinks=true,linkcolor=blue,citecolor=blue,urlcolor=blue,pdftitle={Separable integer partition classes and Slater's list -- II}}

\newtheorem{theorem}{Theorem}[section]
\newtheorem{lemma}[theorem]{Lemma}
\newtheorem{cor}[theorem]{Corollary}

\theoremstyle{definition}

\theoremstyle{remark}
\newtheorem{remark}[theorem]{Remark}

\numberwithin{equation}{section}
\allowdisplaybreaks
\newcommand\numberthis{\addtocounter{equation}{1}\tag{\theequation}}

\begin{document}

\title[Separable integer partition classes and Slater's list -- II]{Separable integer partition classes and Slater's list -- II}

\author{Aritram Dhar}
\address{Department of Mathematics, University of Florida, Gainesville, FL 32611, USA}
\email{aritramdhar@gmail.com, aritramdhar@ufl.edu}

\author{Ankush Goswami}
\address{School of Mathematical and Statistical Sciences, University of Texas Rio Grande Valley, Edinburg, TX 78541, USA}
\email{ankushgoswami3@gmail.com, ankush.goswami@utrgv.edu}

\author{Runqiao Li}
\address{School of Mathematical and Statistical Sciences, University of Texas Rio Grande Valley, Edinburg, TX 78541, USA}
\email{runqiao.li@utrgv.edu}

\date{June 25, 2026}
\subjclass[2020]{05A17, 05A19, 11P81, 11P84, 11F37, 33D15}
\keywords{partitions, overpartitions, separable integer partition classes, Slater's list, partitions with restricted gaps, partitions with positional gaps}

\begin{abstract}
Slater's list of Rogers--Ramanujan type identities remains a central source of striking series--product formulas in the theory of partitions and basic hypergeometric series. Although many of these identities admit elegant analytic proofs through Bailey pairs, Bailey chains, or transformations of basic hypergeometric series, the partition-theoretic meaning of their series sides is often much less apparent. In this paper, which continues the program initiated in  \href{https://arxiv.org/abs/2603.14179}{https://arxiv.org/abs/2603.14179}, we apply Andrews' theory of separable integer partition classes to further identities from Slater's list. We construct a strict overpartition class and two families of overpartitions with positional gap conditions, in which overlining is permitted only at alternating positions. Their multivariate generating functions give natural refinements of the series sides of Slater's identities (12), (28), (29), (47), (48), (50), and (51). We then use Heine-type transformations, a limiting form of Heine's transformations, Watson's $q$-analogue of Whipple's theorem, and classical theta-product identities to obtain alternative series representations and recover the associated products. In addition, we derive a new companion identity to a Slater identity and a signed companion formula. Our results further demonstrate that SIP classes provide a flexible framework for converting basic hypergeometric series into structured partition generating functions, while simultaneously producing refinements, transformations, and new Rogers--Ramanujan type identities.
\end{abstract}

\maketitle
\section{Introduction}

The Rogers--Ramanujan identities and their numerous generalizations form one of the principal meeting points of partition theory, modular forms, and basic hypergeometric series. Beginning with the work of Rogers \cite{Rogers1894}, and continuing through the discoveries of Ramanujan, Bailey, Watson, and many others, these identities have revealed a remarkable phenomenon: a $q$-hypergeometric series with a quadratic exponent frequently coincides with an infinite product having a simple congruence-theoretic meaning. Standard accounts of this subject may be found in the books of Andrews \cite{A}, Andrews and Berndt \cite{AndrewsBerndt2005,AndrewsBerndt2009}, Fine \cite{Fine88}, and Gasper and Rahman \cite{GasperRahman2004}.

A particularly influential compilation is Slater's list of 130 Rogers--Ramanujan type identities \cite{Slater1952}. The identities in this list exhibit a wide range of analytic shapes, including single sums, multisums, theta-type series, and products of eta-quotient type. Their analytic derivations are closely tied to Bailey pairs and related transformation theory; see, for example, Andrews' analytic treatment of the Rogers--Ramanujan identities \cite{Andrews1974} and the discussions in \cite{A,AndrewsBerndt2005,AndrewsBerndt2009,Fine88,GasperRahman2004}. However, an analytic proof of a series--product identity does not automatically explain why the series side should enumerate a natural class of partitions. This problem is especially pronounced when the summand contains factors such as
\[
\frac{(-q;q^2)_n}{(q;q)_{2n}},
\qquad
\frac{(-1;q^2)_n}{(q;q)_{2n}},
\qquad\text{or}\qquad
\frac{(-q^2;q^2)_n}{(q;q)_{2n+1}},
\]
whose partition-theoretic meaning is not immediately visible from the formula itself. The notation and terminology used here and throughout this section are defined in Section \ref{sec:prelim}.

The theory of separable integer partition classes, introduced by Andrews \cite{AndrewsSIP}, provides a systematic mechanism for addressing this issue. Roughly speaking, an SIP class consists of a finite basis of minimal configurations together with a freely extendable tail whose parts are multiples of a fixed modulus. This decomposition converts the generating function of the basis into a $q$-hypergeometric series and, at the same time, preserves a direct partition interpretation. Andrews used this framework to revisit several classical partition theorems, including the first G\"ollnitz--Gordon identity and Schur's partition theorem \cite{AndrewsSIP}. Since then, the SIP point of view has been developed further in a number of directions, including partition classes with congruence restrictions, refined gap conditions, and multivariate generalizations; see \cite{ChenHeTangWei2024,HeHuangLiZhang2025,LiThesis2025}.

The present paper is the second part of a broader study of Slater's list through the SIP framework. In \cite{DGLSIPI}, we considered identities associated with strict partitions having mod-$4$ gap restrictions, strict overpartitions, and partitions satisfying positional gap conditions. The principal theme of that paper was that suitably chosen SIP bases lead naturally to multivariate refinements of otherwise opaque series sides, and that these refinements can be transformed analytically into product identities or alternative hypergeometric expressions. The present paper continues this program with a different collection of identities from Slater's list and with new partition families in which the parity of the \emph{position} of a part plays an essential role.

More precisely, we study the identities listed in Table~\ref{tab:slater-identities}, together with two accompanying formulas that arise naturally from the same SIP constructions. The first accompanying formula is a new companion to a classical Slater identity, while the second is a signed companion obtained from the same mod-$4$ partition family. The main new objects are overpartition classes in which overlining is allowed only at prescribed alternating positions. This positional restriction is intrinsic to the constructions: it produces the factors
\[
(-aq;q^2)_n
\qquad\text{and}\qquad
(-aq^2;q^2)_n
\]
in the multivariate generating functions.

For instance, the series side of Slater's identity (29),
\[
\sum_{n\geq0}\frac{(-q;q^2)_n q^{n^2}}{(q;q)_{2n}},
\]
contains a mixture of an odd-step Pochhammer factor and a denominator of length $2n$, making a direct interpretation far from obvious. We show that it arises naturally from an SIP class of overpartitions in which only odd-indexed parts may be overlined and in which the allowed gap between each odd-indexed part and the following even-indexed part depends on whether the former is overlined. A companion class, with overlining permitted at even-indexed positions, accounts for identities (28), (50), and (51).

The constructions in this paper have three related features. First, each partition family has an explicitly described basis, and the corresponding SIP decomposition gives a multivariate generating function in which the auxiliary variables record natural statistics such as the number of overlined parts, the number of odd-indexed or even-indexed positions, and the numbers of odd and even parts. Second, these refined series admit useful transformations. A limiting form of Heine's transformations gives a convenient two-parameter transformation that repeatedly converts the SIP series into alternative hypergeometric forms. In selected cases, Watson's $q$-analogue of Whipple's theorem, Jacobi's triple product identity, Euler's pentagonal number theorem, and the quintuple product identity reduce the transformed expressions to infinite products. Third, the refined formulas reveal companion identities which are not apparent from Slater's original list. Thus, the present paper not only interprets further entries of Slater's list, but also shows how the SIP method can generate new identities from previously constructed partition classes.

The remainder of the paper is organized as follows. Section~\ref{sec:prelim} records the SIP framework and the analytic transformations used throughout. In Section~\ref{sec:strictGG}, we treat Slater's identity (12) through strict overpartitions and obtain two new G\"ollnitz-Gordon type companions. We study these new companions through a shifted mod-$4$ partition class, develop its signed analogue, and record the resulting partition-theoretic consequences. In Section~\ref{sec:S1}, we introduce an SIP class of overpartitions with admissible overlines at even-indexed positions and derive identities (28), (50), and (51), together with several companion transformations. Section~\ref{sec:S2} develops the corresponding odd-indexed overpartition class and obtains identities (29), (47), and (48). 

\begin{center}
\small
\renewcommand{\arraystretch}{1.8}
\begin{longtable}{>{\raggedright\arraybackslash}p{12.55cm}>{\centering\arraybackslash}p{2.7cm}}
\caption{Identities studied in this paper. The first seven occur in Slater's list. The label \textup{C28} denotes the companion to Slater's identity (28), while \textup{NGG2} and \textup{NGG2$'$} denote the new G\"ollnitz--Gordon companions.}\label{tab:slater-identities}\\
\toprule
\textbf{Identity} & \textbf{Source / label}\\
\midrule
\endfirsthead
\multicolumn{2}{c}{\tablename~\thetable\ (continued)}\\
\toprule
\textbf{Identity} & \textbf{Source / label}\\
\midrule
\endhead
\bottomrule
\endfoot
$\displaystyle
\sum_{n=0}^{\infty}\frac{(-1;q)_nq^{\frac{n(n+1)}{2}}}{(q;q)_{n}}
=\frac{(-q;q)_{\infty}(q^2,q^2,q^4;q^4)_{\infty}}{(q;q)_{\infty}}
$
& (12)\\
$\displaystyle
\sum_{n=0}^{\infty}\frac{(-q;q^2)_{n}q^{n^2+2n}}{(q^4;q^4)_{n}}
=\frac{(q^6;q^{12})_{\infty}}{(q^3,q^4,q^8,q^9;q^{12})_{\infty}}
$
& (NGG2)\\
$\displaystyle
\sum_{n=0}^{\infty}\frac{(-1)^n(q;q^2)_nq^{n^2+2n}}{(q^4;q^4)_n}
=\frac{(q^3,q^9;q^{12})_{\infty}}{(q^4,q^8;q^{12})_{\infty}}
$
& (NGG2$'$)\\
$\displaystyle
\sum_{n=0}^{\infty}\frac{(-q^2;q^2)_{n}q^{n(n+1)}}{(q;q)_{2n+1}}
=\frac{(-q^2;q^2)_{\infty}(-q,-q^5,q^6;q^6)_{\infty}}{(q^2;q^2)_{\infty}}
$
& (28)\\
$\displaystyle
\sum_{n=0}^{\infty}\frac{(-q;q^2)_{n}q^{n^2+n}}{(q;q)_{2n+1}}
=\frac{(-q^2;q^2)_{\infty}}{(q;q^2)_{\infty}}
=\frac{(q^4;q^4)_{\infty}}{(q;q)_{\infty}}
$
& (C28)\\
$\displaystyle
\sum_{n=0}^{\infty}\frac{(-q;q^2)_{n}q^{n^2}}{(q;q)_{2n}}
=\frac{(-q;q^2)_{\infty}(-q^2,-q^4,q^6;q^6)_{\infty}}{(q^2;q^2)_{\infty}}
$
& (29)\\
$\displaystyle
\sum_{n=0}^{\infty}\frac{(-1;q^2)_nq^{n^2}}{(q;q)_{2n}}
=\frac{(q^5,q^7,q^{12};q^{12})_{\infty}+q(q,q^{11},q^{12};q^{12})_{\infty}}{(q;q)_{\infty}}
$
& (47)\\
$\displaystyle
\sum_{n=0}^{\infty}\frac{(-1;q^2)_nq^{n(n+1)}}{(q;q)_{2n}}
=\frac{(q^5,q^7,q^{12};q^{12})_{\infty}-q(q,q^{11},q^{12};q^{12})_{\infty}}{(q;q)_{\infty}}
$
& (48)\\
$\displaystyle
\sum_{n=0}^{\infty}\frac{(-q;q^2)_nq^{n(n+2)}}{(q;q)_{2n+1}}
=\frac{(q^2,q^{10},q^{12};q^{12})_{\infty}}{(q;q)_{\infty}}
$
& (50)\\
$\displaystyle
\sum_{n=0}^{\infty}\frac{(-q;q^2)_nq^{n(n+1)}}{(q;q)_{2n+1}}
=\frac{(q^4,q^8,q^{12};q^{12})_{\infty}}{(q;q)_{\infty}}
$
& (51)\\
\end{longtable}
\end{center}

\section{Preliminaries}\label{sec:prelim}

Throughout, $|q|<1$. For $n\in\mathbb Z_{\geq0}\cup\{\infty\}$, we use the standard notation
\[
(a)_n = (a;q)_n:=\prod_{j=0}^{n-1}(1-aq^j),
\qquad
(a_1,\ldots,a_r)_n = (a_1,\ldots,a_r;q)_n:=\prod_{j=1}^r(a_j;q)_n.
\]
When the base is clear from the context, we write $(a)_n$ in place of $(a;q)_n$. We also use the Gaussian polynomial
\[
{m\brack r}_q:=\frac{(q;q)_m}{(q;q)_r(q;q)_{m-r}}
\qquad (0\le r\le m).
\]

A \textit{partition} $\lambda$ is a non-increasing finite sequence $\lambda = (\lambda_1,\lambda_2,\dots,\lambda_k)$ of positive integers. The elements $\lambda_i$ appearing in the sequence $\lambda$ are called the \textit{parts} of $\lambda$. The number of parts or \textit{length} of $\lambda$ is denoted by $\#(\lambda)$. The sum of all the parts of $\lambda$ is called the \textit{size} of $\lambda$ and denoted by $|\lambda|$. We say $\lambda$ is a partition of $n$ if its size is equal to $n$.  An overpartition \cite{overpartitions} is a partition $\lambda$ where the last (or first) appearance of each part may be overlined, the number of overlined parts is denoted by $a(\lambda)$.

We briefly recall the SIP framework introduced by Andrews \cite{AndrewsSIP}. Let $\mathcal P$ be a class of partitions and let $k$ be a positive integer. A subset $\mathcal B\subseteq\mathcal P$ is called a \emph{basis} for an SIP class of modulus $k$ if every partition $\lambda=(\lambda_1,\ldots,\lambda_n)\in\mathcal P$ has a unique decomposition
\[
\lambda_i=b_i+\pi_i \qquad (1\le i\le n),
\]
where $(b_1,\ldots,b_n)\in\mathcal B$ and $(\pi_1,\ldots,\pi_n)$ is a partition into nonnegative parts divisible by $k$, and conversely every such pair produces a partition in $\mathcal P$. If $B(n;q)$ denotes the generating function of basis partitions with $n$ parts, then
\begin{equation}\label{eq:sip-general-gf}
\sum_{\lambda\in\mathcal P}q^{|\lambda|}
=\sum_{n\ge0}\frac{B(n;q)}{(q^k;q^k)_n}.
\end{equation}
The same construction accommodates natural auxiliary statistics by incorporating their weights in $B(n;q)$. For a Young-diagram description of the basis--tail decomposition and further examples of its use in the study of Slater-type identities, we refer the reader to the first paper of this series \cite{DGLSIPI}.

\begin{theorem}[{Heine's transformations \cite[(1.4.4)--(1.4.6)]{GasperRahman2004}}]
\begin{align}
\sum_{n=0}^{\infty}\frac{(a)_n(b)_n}{(q)_n(c)_n}z^n&=\frac{(b,za;q)_{\infty}}{(c,z;q)_{\infty}}\sum_{n=0}^{\infty}\frac{(c/b)_n(z)_n}{(q)_n(az)_n}b^{n}.\label{eq:Heine1}\\
&=\frac{(c/b, bz;q)_{\infty}}{(c,z;q)_{\infty}}\sum_{n=0}^{\infty}\frac{(abz/c)_n(b)_n}{(q)_n(bz)_n}\left(\frac{c}{b}\right)^{n}\label{Heine2}\\
&=\frac{(abz/c;q)_{\infty}}{(z;q)_{\infty}}\sum_{n=0}^{\infty}\frac{(c/a)_n(c/b)_n}{(q)_n(c)_n}\left(\frac{abz}{c}\right)^{n}.\label{Heine3}
\end{align}
\end{theorem}
We next obtain the following useful result.
\begin{lemma}\label{lem1}
We have
\begin{align*}
 \sum_{n\ge 0}\dfrac{(A)_n}{(q)_n(B)_n}(-1)^n x^n q^{\frac{n(n-1)}{2}}&=\dfrac{(x)_\infty}{(B)_\infty}\sum_{n\ge 0}\dfrac{(Ax/B)_n}{(q)_n(x)_n}(-1)^nq^{n(n-1)/2}B^n.\\
 &=(Ax/B)_\infty\sum_{n\ge 0}\dfrac{(B/A)_n}{(q)_n(B)_n}\left(\dfrac{Ax}{B}\right)^n
 \end{align*}
\end{lemma}
\begin{proof}
 We let $a=A,\,b=\rho,\,z\to z/\rho,\,c=B$ and $\rho\to\infty$ and set $z=x$ in \eqref{Heine2} to get the first equality and use the same substitutions in \eqref{Heine3} to get the second equality.
\end{proof}
We then state a finite version of $q$-binomial theorem from \cite[p. $36$, $(3.3.6)$]{A}, which is as follows.
\begin{equation}\label{eq:FQbinom}
(z;q)_{n}=\sum_{k=0}^{n}(-1)^kz^{k}q^{\binom{k}{2}}{n\brack k}_{q}.
\end{equation}
We next state Watson's $q$-analogue of Whipple's theorem \cite[p. $81$, $(4.1.3)$]{AndrewsBerndt2009}. If $\alpha$, $\beta$, $\gamma$, $\delta$, and $\epsilon$ are complex numbers such that $\beta\gamma\delta\epsilon\neq 0$, and if $N$ is any non-negative integer, then
\begin{align*}
\sum\limits_{n=0}^{\infty}&\dfrac{(\alpha,q\sqrt{\alpha},-q\sqrt{\alpha},\beta,\gamma,\delta,\epsilon,q^{-N};q)_n}{(q,\sqrt{\alpha},-\sqrt{\alpha},\frac{\alpha q}{\beta},\frac{\alpha q}{\gamma},\frac{\alpha q}{\delta},\frac{\alpha q}{\epsilon},\alpha q^{N+1};q)_n}\left(\frac{\alpha^2q^{N+2}}{\beta\gamma\delta\epsilon}\right)^n\\
&= \dfrac{(\alpha q,\frac{\alpha q}{\delta\epsilon};q)_N}{(\frac{\alpha q}{\delta},\frac{\alpha q}{\epsilon};q)_N}\sum\limits_{n=0}^{\infty}\dfrac{(\frac{\alpha q}{\beta\gamma},\delta,\epsilon,q^{-N};q)_n}{(q,\frac{\alpha q}{\beta},\frac{\alpha q}{\gamma},\frac{\delta\epsilon q^{-N}}{\alpha};q)_n}q^n.\numberthis\label{eq23}
\end{align*}
Let $\gamma,\epsilon$, and $N$ all go to infinity, we get a handy identity as follows.
\begin{equation}\label{eq:mastersubs}
\sum_{n\geq0}\frac{(\alpha,q\sqrt{\alpha},-q\sqrt{\alpha},\beta,\delta;q)_{n}}{(q,\sqrt{\alpha},-\sqrt{\alpha},\frac{\alpha q}{\beta},\frac{\alpha q}{\delta};q)_n}\left(-\frac{\alpha^2}{\beta\delta}\right)^nq^{\frac{n(3n+1)}{2}}=\frac{(\alpha q;q)_{\infty}}{(\frac{\alpha q}{\delta};q)_{\infty}}\sum_{n\geq0}\frac{(\delta;q)_{n}}{(q,\frac{\alpha q}{\beta};q)_{n}}\left(-\frac{\alpha}{\delta}\right)^nq^{\frac{n(n+1)}{2}}
\end{equation}

\section{Strict Partition Classes and G\"ollnitz--Gordon Type Companions}\label{sec:strictGG}

\subsection{Strict overpartitions and Slater's identity (12)}\label{sec:slater12}

Let $\overline{\mathcal{S}}$ be the set of partitions satisfying the following conditions.
\begin{enumerate}
    \item $\lambda$ is a strict partition;
    \item $\lambda_i$ may be overlined if $\lambda_{i}\geq2$ and $\lambda_{i}-\lambda_{i+1}\geq2$.
\end{enumerate}
We call partitions in $\overline{\mathcal{S}}$ strict overpartitions, and let $\overline{\#}(\lambda)$ be the number of overlined parts in $\lambda$. The following refined generating function was established in the first paper of this series \cite{DGLSIPI}.
\begin{theorem}[Dhar--Goswami--Li \cite{DGLSIPI}]\label{thm:GFStrictOverpartition}We have
\begin{align*}
\sum_{\lambda\in\overline{\mathcal{S}}}a^{\overline{\#}(\lambda)}z^{\#(\lambda)}q^{|\lambda|}=\sum_{n=0}^{\infty}\frac{(-aq;q)_{n}}{(q;q)_{n}}z^nq^{\frac{n(n+1)}{2}}=(-aq)_\infty (-zq)_\infty\sum_{n\ge 0}\dfrac{(-aq)^n}{(q)_n(-zq)_n}.\numberthis\label{eq31}
\end{align*}
\end{theorem}
We now recover Slater's identity (12).
\begin{theorem}\label{thm:Slater12} We have
\begin{align*}
\sum\limits_{n=0}^{\infty}\frac{(-1;q)_nq^{\frac{n(n+1)}{2}}}{(q;q)_n}=\frac{(-q;q)_{\infty}(q^2,q^2,q^4;q^4)_{\infty}}{(q;q)_{\infty}}.\numberthis\label{eq32}
\end{align*}
\end{theorem}
\begin{proof}
Letting $a=1/q$ and $z=1$ in \eqref{eq32}, we get
\begin{align*}
\sum\limits_{n=0}^{\infty}\frac{(-1;q)_nq^{\frac{n(n+1)}{2}}}{(q;q)_n} &= (-1;q)_{\infty}(-q;q)_{\infty}\sum\limits_{n=0}^{\infty}\frac{(-1)^n}{(q;q)_n(-q;q)_n}\\
&= (-1;q)_{\infty}(-q;q)_{\infty}\sum\limits_{n=0}^{\infty}\frac{(-1)^n}{(q^2;q^2)_n}.\numberthis\label{eq33}
\end{align*}
Using Cauchy's identity \cite[p. $17$, $(2.2.1)$]{A} on the right-hand side of \eqref{eq33}, we obtain
\begin{align*}
\sum\limits_{n=0}^{\infty}\frac{(-1;q)_nq^{\frac{n(n+1)}{2}}}{(q;q)_n} &= \frac{(-1;q)_{\infty}(-q;q)_{\infty}}{(-1;q^2)_{\infty}}\\
&= \frac{(-q;q)^2_{\infty}}{(-q^2;q^2)_{\infty}}\\
&= \frac{(-q;q)_{\infty}(q^2;q^4)_{\infty}}{(q;q^2)_{\infty}}\\
&= \frac{(-q;q)_{\infty}(q^2;q^4)_{\infty}(q^2;q^2)_{\infty}}{(q;q^2)_{\infty}(q^2;q^2)_{\infty}}\\
&= \frac{(-q;q)_{\infty}(q^2,q^2,q^4;q^4)_{\infty}}{(q;q)_{\infty}}.\\
\end{align*}
This completes the proof.
\end{proof}

\subsection{A shifted mod-4 class and new companion identities}\label{sec:ngg}
The following identity is Equation $(25)$ in Slater's list, and was treated in \cite{DGLSIPI}.
\begin{equation}\label{eq:NGG1}
\sum_{n=0}^{\infty}\frac{(-q;q^2)_{n}q^{n^2}}{(q^4;q^4)_{n}}
=\dfrac{(q^3;q^6)^2_\infty(q^6;q^6)_\infty(-q;q^2)_\infty}{(q^2;q^2)_\infty}.
\end{equation}
Here we will prove the following companions, which are not included in Slater's list.
\begin{equation}\label{eq:NGG2}
\sum_{n=0}^{\infty}\frac{(-q;q^2)_{n}q^{n^2+2n}}{(q^4;q^4)_{n}}
=\frac{(q^6;q^{12})_{\infty}}{(q^3,q^4,q^8,q^9;q^{12})_{\infty}}, 
\end{equation}
\begin{equation}
\sum_{n=0}^{\infty}\frac{(-1)^n(q;q^2)_nq^{n^2+2n}}{(q^4;q^4)_n}
=\frac{(q^3,q^9;q^{12})_{\infty}}{(q^4,q^8;q^{12})_{\infty}}.
\end{equation}
Recall that the G\"ollnitz-Gordon identities, also Equation $(36)$ and Equation $(34)$ in Slater's list, are stated as
\begin{equation}\label{eq:GG1}
\sum_{n=0}^{\infty}\frac{(-q;q^2)_nq^{n^{2}}}{(q^2;q^2)_n}=\frac{1}{(q,q^4,q^7;q^8)_{\infty}},
\end{equation}
\begin{equation}
\sum_{n=0}^{\infty}\frac{(-q;q^2)_nq^{n^{2}+2n}}{(q^2;q^2)_n}=\frac{1}{(q^3,q^4,q^5;q^8)_{\infty}}.
\end{equation}
Due to the similarity of the sum sides, we call \eqref{eq:NGG1} and \eqref{eq:NGG2} the new G\"ollnitz-Gordon identities.

Let $\mathcal{P}_4$ be the set of strict partitions satisfying the following gap conditions.
\begin{enumerate}
    \item $\lambda_i-\lambda_{i+1}\equiv 2\pmod{4}$ if they are both odd;
    \item $\lambda_i-\lambda_{i+1}\equiv 0\pmod{4}$ if they are both even;
    \item $\lambda_{i}-\lambda_{i+1}\equiv 3\pmod{4}$ if they have different parity;
    \item $\lambda_{i}\equiv 1\ \text{or}\ 2\pmod{4}$ if it is the smallest part.
\end{enumerate}
Let $o(\lambda)$ and $e(\lambda)$ denote the numbers of odd and even parts of $\lambda$, respectively. The following refinement was proved in \cite{DGLSIPI}.
\begin{theorem}[Dhar--Goswami--Li \cite{DGLSIPI}]\label{thm:P4Gen}
\begin{align*}
\sum_{\lambda\in\mathcal{P}_4}x^{o(\lambda)}y^{e(\lambda)}q^{|\lambda|} &= \sum_{n=0}^{\infty}\frac{(-yq/x;q^2)_{n}x^{n}q^{n^2}}{(q^4;q^4)_{n}}\\
&= \frac{(-xq;q^2)_{\infty}}{(yq^2;q^2)_{\infty}}
\Biggl(1+(1+y)\sum_{n=1}^{\infty}
\frac{(-1)^n(1-yq^{4n})(y^2q^4;q^4)_{n-1}}{(q^4;q^4)_n}\\
&\hspace{3.5cm}\times
\frac{(-yq/x;q^2)_n x^n q^{3n^2}}{(-xq;q^2)_n}\Biggr).
\numberthis\label{eq:P4eq1}
\end{align*}
\end{theorem}
We now prove the new G\"ollnitz--Gordon companion identity to Slater's identity (25).
\begin{theorem}
\begin{align*}
\sum_{n=0}^{\infty}\frac{(-q;q^2)_{n}q^{n^2+2n}}{(q^4;q^4)_{n}}=\frac{(q^6;q^{12})_{\infty}}{(q^3,q^4,q^8,q^9;q^{12})_{\infty}}.\numberthis\label{eq:P4eq2}
\end{align*}
\end{theorem}
\begin{proof}
Letting $x=y=q^2$ in \eqref{eq:P4eq1}, we get
\begin{align*}
\sum_{n=0}^{\infty}\frac{(-q;q^2)_{n}q^{n^2+2n}}{(q^4;q^4)_{n}} &= \frac{(-q^3;q^2)_{\infty}}{(q^4;q^2)_{\infty}}\left(1+(1+q^2)\sum\limits_{n=1}^{\infty}\frac{(-1)^n(1-q^{4n+2})(q^8;q^4)_{n-1}(-q;q^2)_nq^{3n^2+2n}}{(q^4;q^4)_n(-q^3;q^2)_n}\right)\\
&= \frac{(-q^3;q^2)_{\infty}}{(q^4;q^2)_{\infty}}\left(1+\frac{1}{1-q}\sum\limits_{n=1}^{\infty}(-1)^n(1-q^{2n+1})q^{3n^2+2n}\right)\\
&= \frac{(-q;q^2)_{\infty}}{(q^2;q^2)_{\infty}}\sum\limits_{n=0}^{\infty}(-1)^n(1-q^{2n+1})q^{3n^2+2n}\\
&= \frac{(-q;q^2)_{\infty}}{(q^2;q^2)_{\infty}}\sum\limits_{n=-\infty}^{\infty}(-1)^nq^{3n^2+2n}\\
&= \frac{(-q;q^2)_{\infty}(q,q^5,q^6;q^6)_{\infty}}{(q^2;q^2)_{\infty}}\numberthis\label{eq:P4eq3}\\
&= \frac{(q^6;q^{12})_{\infty}}{(q^3,q^4,q^8,q^9;q^{12})_{\infty}}
\end{align*}
where \eqref{eq:P4eq3} follows by replacing $q\rightarrow q^3$ and $z=-q^2$ in Jacobi's Triple Product Identity.
\end{proof}
Here we also give a partition interpretation for the series in \eqref{eq:P4eq2}. Let $\mathcal{P}'_4$ be the set of strict partitions such that each pair of consecutive parts satisfies the same restriction as in $\mathcal{P}_4$, and the smallest part is congruent to either $3$ or $0$ modulo $4$.
\begin{theorem}\label{thm:P4primeGF}
The generating function for $\mathcal{P}'_4$ is
\begin{equation}\label{eq:P4primeGF}
\sum_{\lambda\in\mathcal{P}'_4}x^{o(\lambda)}y^{e(\lambda)}q^{|\lambda|}=\sum_{n=0}^{\infty}\frac{(-yq/x;q^2)_{n}x^nq^{n^2+2n}}{(q^4;q^4)_{n}}.
\end{equation}
\end{theorem}

\begin{proof}
Note that to get the right-hand side of \eqref{eq:P4primeGF}, we need to set $x\to xq^2$ and $y\to yq^2$ in \eqref{eq:P4eq1}. Since $x$ and $y$ keep track of the numbers of odd and even parts, respectively, this substitution means adding $2$ to each part of the partitions in $\mathcal{P}_4$. Consequently, we end up with a partition in $\mathcal{P}'_4$.
\end{proof}

\begin{remark}
It is straightforward to verify that $\mathcal{P}'_4$ is an SIP class of modulus $4$. Thus, by the same SIP argument used for $\mathcal{P}_4$ in \cite{DGLSIPI}, one can also prove Theorem~\ref{thm:P4primeGF} directly from the basis description.
\end{remark}
We conclude with the signed new G\"ollnitz--Gordon companion associated with $\mathcal{P}'_4$.
\begin{theorem}
\begin{equation}\label{eq:P4primeSigned}
\sum_{\lambda\in\mathcal{P}'_4}(-1)^{o(\lambda)}q^{|\lambda|}=\sum_{n=0}^{\infty}\frac{(q;q^2)_{n}(-1)^nq^{n^2+2n}}{(q^4;q^4)_{n}}=\frac{(q^3,q^9;q^{12})_{\infty}}{(q^4,q^8;q^{12})_{\infty}}.
\end{equation}
\end{theorem}

\begin{proof}
Set $x\to -q^2$ and $y\to q^2$ in \eqref{eq:P4eq1} to get
\begin{align*}
\sum\limits_{n=0}^{\infty}\frac{(q;q^2)_n(-1)^nq^{n^2+2n}}{(q^4;q^4)_n} &= \frac{(q^3;q^2)_{\infty}}{(q^4;q^2)_{\infty}}\left(1+\frac{1}{1+q}\sum\limits_{n=1}^{\infty}\left(q^{3n^2+2n}+q^{3n^2+4n+1}\right)\right)\\
&= \frac{(q;q^2)_{\infty}}{(q^2;q^2)_{\infty}}\sum\limits_{n=-\infty}^{\infty}q^{3n^2+2n}\\
&= \frac{(q;q^2)_{\infty}(-q,-q^5,q^6;q^6)_{\infty}}{(q^2;q^2)_{\infty}},
\end{align*}
where the last line above follows using Jacobi's Triple Product identity. Thus, we have
\begin{align*}
\sum\limits_{n=0}^{\infty}\frac{(q;q^2)_n(-1)^nq^{n^2+2n}}{(q^4;q^4)_n} = \frac{(q^3,q^9;q^{12})_{\infty}}{(q^4,q^8;q^{12})_{\infty}}.
\end{align*}
\end{proof}

It is also easy to see that replacing $q\rightarrow -q$ in \eqref{eq:P4eq2} gives us \eqref{eq:P4primeSigned} which then gives us the following new partition identity.
\begin{theorem}
We have
\begin{align*}
\sum\limits_{\lambda\in\mathcal{P}_4}(-1)^{|\lambda|}q^{|\lambda|+2\#(\lambda)} = \sum\limits_{\lambda\in\mathcal{P}'_4}(-1)^{o(\lambda)}q^{|\lambda|}.
\end{align*}
\end{theorem}
\begin{proof}
For any partition $\lambda\in\mathcal{P}_4$, $+2\#(\lambda)$ means adding $2$ to each part of $\lambda$ which produces a partition in $\mathcal{P}'_4$ and the parity of $|\lambda|$ is the same as the parity of $o(\lambda)$, hence the result.
\end{proof}

\section{Even-position overpartitions and identities (28), (50), and (51)}\label{sec:S1}
In this section, we treat Equation $(28)$, $(50)$, and $(51)$ in Slater's list. They are given by
\begin{equation}\label{eq:Slater28}
\sum_{n=0}^{\infty}\frac{(-q^2;q^2)_{n}q^{n(n+1)}}{(q;q)_{2n+1}}
=\dfrac{(-q^2;q^2)_{\infty}(-q,-q^5,q^6;q^6)_{\infty}}{(q^2;q^2)_{\infty}},
\end{equation}
\begin{equation}\label{eq:Slater50}
\sum_{n=0}^{\infty}\frac{(-q;q^2)_nq^{n(n+2)}}{(q;q)_{2n+1}}
=\frac{(q^2,q^{10},q^{12};q^{12})_{\infty}}{(q;q)_{\infty}},
\end{equation}
and
\begin{equation}\label{eq:Slater51}
\sum_{n=0}^{\infty}\frac{(-q;q^2)_nq^{n(n+1)}}{(q;q)_{2n+1}}
=\frac{(q^4,q^8,q^{12};q^{12})_{\infty}}{(q;q)_{\infty}}.
\end{equation}

Let $\mathcal{S}_1$ be the set of overpartitions $\lambda=(\lambda_1,\lambda_2,\ldots,\lambda_{\ell})$ such that
\begin{enumerate}
    \item Only even-indexed parts may be overlined.
    \item $\lambda_{2i}-\lambda_{2i+1}\geq1$, and $\lambda_{2i}-\lambda_{2i+1}\geq2$ if $\lambda_{2i}$ is overlined.
\end{enumerate}
The class $\mathcal{S}_1$ is an SIP class of modulus $1$, and the basis $\mathcal{B}_{\mathcal{S}_1}$ is the set of partitions in $\mathcal{S}_1$ such that
\begin{enumerate}
    \item $\lambda_{2i-1}=\lambda_{2i}$ if $\lambda_{2i}>0$.
    \item $\lambda_{2i}-\lambda_{2i+1}=2$ if $\lambda_{2i}$ is overlined, and $\lambda_{2i}-\lambda_{2i+1}=1$ otherwise.
\end{enumerate}
Let $B_{\mathcal S_1}(n,h):=B_{\mathcal S_1}(n,h;a,z,q)$ be the generating function for partitions in $\mathcal{B}_{\mathcal{S}_1}$ with length $n$ and largest part $h$, where $a$, $z$, and $q$ keep track of the number of overlined parts, the number of even-indexed parts, and the weight, respectively. We will show that
$$\sum_{\lambda\in\mathcal{S}_1}a^{\overline{\#}(\lambda)}z^{\lfloor\frac{\#(\lambda)}{2}\rfloor}q^{|\lambda|}=\sum_{n=0}^{\infty}\frac{(-aq^2;q^2)_{n}z^nq^{n(n+1)}}{(q;q)_{2n+1}}.$$
Let $(a,z)\to(1,1)$, $(a,z)\to(q^{-1},1)$, and $(a,z)\to(q^{-1},q)$, we will have the sum sides of \eqref{eq:Slater28}, \eqref{eq:Slater50}, and \eqref{eq:Slater51}, respectively.

\begin{theorem}[Initial values for the $\mathcal S_1$ basis]\label{thm:S1-initial}
The initial values for $B_{\mathcal S_1}(n,h)$ are
\begin{equation*}
    B_{\mathcal S_1}(0,h)=\begin{cases}
        1 &  \text{if $h=0$,}\\
        0 &  \text{otherwise.}
    \end{cases}
\end{equation*}
\begin{equation*}
    B_{\mathcal S_1}(1,h)=\begin{cases}
        q &  \text{if $h=1$,}\\
        0 &  \text{otherwise.}
            \end{cases}
\end{equation*}
\end{theorem}
\begin{theorem}[Basis recurrence for $\mathcal S_1$]\label{thm:S1-recurrence}
For any $n,h\geq1$, $B_{\mathcal S_1}(n,h)$ satisfies the following recurrence.
\begin{equation}\label{eq:S1Rec}
B_{\mathcal S_1}(n,h)=zq^{2h}B_{\mathcal S_1}(n-2,h-1)+azq^{2h}B_{\mathcal S_1}(n-2,h-2).
\end{equation}
\end{theorem}

\begin{proof}
Given a partition $\lambda$ in $\mathcal{B}_{\mathcal{S}_1}$ with length $n$ and largest part $h$, then the second largest part of $\lambda$ can be either $h$ or $\overline{h}$. In the first case, the third largest part of $\lambda$ will be $h-1$ and the weight of the first two parts is $zq^{2h}$, while in the second case, the third largest part of $\lambda$ is $h-2$, and the weight of the first two parts is $azq^{2h}$. So, by deleting the first two parts from $\lambda$, we have the desired recurrence.
\end{proof}

\begin{theorem}[Closed basis forms for $\mathcal S_1$]\label{thm:S1-closed}
The closed forms for $B_{\mathcal S_1}(n,h)$ are given by
\begin{equation}\label{eq:S1B1}
B_{\mathcal S_1}(2n,n+h)=a^hz^nq^{n(n+1)+h(h+1)}{n\brack h}_{q^2},
\end{equation}
and
\begin{equation}\label{eq:S1B2}
B_{\mathcal S_1}(2n+1,n+h+1)=a^hz^nq^{(n+1)^2+n+h(h+1)}{n\brack h}_{q^2}.
\end{equation}
\end{theorem}

\begin{proof}
We start with \eqref{eq:S1B1}, note that
\begin{align*}
B_{\mathcal S_1}(2n,n+h)=&a^hz^nq^{n(n+1)+h(h+1)}{n\brack h}_{q^2}\\
=&a^hz^nq^{n(n+1)+h(h+1)}\left(q^{2h}{n-1\brack h}_{q^2}+{n-1\brack h-1}_{q^2}\right)\\
=&a^hz^nq^{n(n+1)+h(h+1)+2h}{n-1\brack h}_{q^2}+a^hz^nq^{n(n+1)+h(h+1)}{n-1\brack h-1}_{q^2}\\
=&zq^{2n+2h}\cdot a^hz^{n-1}q^{n(n-1)+h(h+1)}{n-1\brack h}_{q^2}\\
&+azq^{2n+2h}\cdot a^{h-1}z^{n-1}q^{n(n-1)+h(h-1)}{n-1\brack h-1}_{q^2}\\
=&zq^{2n+2h}B_{\mathcal S_1}(2n-2,n+h-1)+zaq^{2n+2h}B_{\mathcal S_1}(2n-2,n+h-2),
\end{align*}
so \eqref{eq:S1B1} satisfies \eqref{eq:S1Rec}. Similarly, for \eqref{eq:S1B2},
\begin{align*}
B_{\mathcal S_1}(2n+1,n+h+1)=&a^hz^nq^{(n+1)^2+n+h(h+1)}{n\brack h}_{q^2}\\
=&a^hz^nq^{(n+1)^2+n+h(h+1)}\left(q^{2h}{n-1\brack h}_{q^2}+{n-1\brack h-1}_{q^2}\right)\\
=&a^hz^nq^{(n+1)^2+n+h(h+1)+2h}{n-1\brack h}_{q^2}+a^hz^nq^{(n+1)^2+n+h(h+1)}{n-1\brack h-1}_{q^2}\\
=&zq^{2n+2h+2}\cdot a^hz^{n-1}q^{n^2+n+h(h+1)-1}{n-1\brack h}_{q^2}\\
&+azq^{2n+2h+2}\cdot a^{h-1}z^{n-1}q^{n^2+n+h(h-1)-1}{n-1\brack h-1}_{q^2}\\
=&zq^{2n+2h+2}B_{\mathcal S_1}(2n-1,n+h)+zaq^{2n+2h+2}B_{\mathcal S_1}(2n-1,n+h-1).
\end{align*}
So, they both satisfy the desired recurrence. By checking the initial values, we finish the proof.
\end{proof}

\begin{theorem}\label{thm:S1}
The following three-variable generating function holds for $\mathcal{S}_1$.
\begin{equation}\label{eq:azS1}
\sum_{\lambda\in\mathcal{S}_1}a^{\overline{\#}(\lambda)}z^{\lfloor\frac{\#(\lambda)}{2}\rfloor}q^{|\lambda|}=\sum_{n=0}^{\infty}\frac{(-aq^2;q^2)_{n}z^nq^{n(n+1)}}{(q;q)_{2n+1}}.
\end{equation}
\end{theorem}

\begin{proof}
Since $\mathcal{S}_1$ is an SIP class of modulus $1$,
\begin{align*}
\sum_{\lambda\in\mathcal{S}_1}a^{\overline{\#}(\lambda)}z^{\lfloor\frac{\#(\lambda)}{2}\rfloor}q^{|\lambda|}=&\sum_{n=0}^{\infty}\sum_{h=0}^{\infty}\frac{B_{\mathcal S_1}(n,h)}{(q;q)_{n}}\\
=&\sum_{n=0}^{\infty}\sum_{h=0}^{\infty}\frac{B_{\mathcal S_1}(2n,n+h)}{(q;q)_{2n}}+\sum_{n=0}^{\infty}\sum_{h=0}^{\infty}\frac{B_{\mathcal S_1}(2n+1,n+h+1)}{(q;q)_{2n+1}}\\
=&\sum_{n=0}^{\infty}\frac{z^nq^{n(n+1)}}{(q;q)_{2n}}\sum_{h=0}^{\infty}a^hq^{h(h+1)}{n\brack h}_{q^2}+\sum_{n=0}^{\infty}\frac{z^nq^{(n+1)^2+n}}{(q;q)_{2n+1}}\sum_{h=0}^{\infty}a^hq^{h(h+1)}{n\brack h}_{q^2}\\
=&\sum_{n=0}^{\infty}\frac{(-aq^2;q^2)_{n}z^nq^{n(n+1)}}{(q;q)_{2n}}+\sum_{n=0}^{\infty}\frac{(-aq^2;q^2)_{n}z^nq^{(n+1)^2+n}}{(q;q)_{2n+1}}\\
=&\sum_{n=0}^{\infty}\frac{(-aq^2;q^2)_{n}z^nq^{n(n+1)}}{(q;q)_{2n}}\left(1+\frac{q^{2n+1}}{1-q^{2n+1}}\right)\\
=&\sum_{n=0}^{\infty}\frac{(-aq^2;q^2)_{n}z^nq^{n(n+1)}}{(q;q)_{2n+1}}.
\end{align*}
This completes the proof.
\end{proof}
Note that the two double sums in the second line of the proof are counting partitions in $\mathcal{S}_1$ with length being even and odd, respectively. So, we have also proved the following generating functions.
\begin{cor}\label{cor:S1-evenodd}
Let $\mathcal{S}_{1,e}$ and $\mathcal{S}_{1,o}$ be the set of partitions in $\mathcal{S}_{1}$ with length being even and odd, respectively. Then,
\begin{equation}
\sum_{\lambda\in\mathcal{S}_{1,e}}a^{\overline{\#}(\lambda)}z^{\lfloor\frac{\#(\lambda)}{2}\rfloor}q^{|\lambda|}=\sum_{n=0}^{\infty}\frac{(-aq^2;q^2)_{n}z^nq^{n(n+1)}}{(q;q)_{2n}}
\end{equation}
and
\begin{equation}
\sum_{\lambda\in\mathcal{S}_{1,o}}a^{\overline{\#}(\lambda)}z^{\lfloor\frac{\#(\lambda)}{2}\rfloor}q^{|\lambda|}=\sum_{n=0}^{\infty}\frac{(-aq^2;q^2)_{n}z^nq^{(n+1)^2+n}}{(q;q)_{2n+1}}.
\end{equation}
\end{cor}

\subsection{Transformations and product evaluations}
Next, we consider sum-product identities related to $\mathcal{S}_1$.
\begin{theorem} We have
\begin{equation}
\sum_{n=0}^{\infty}\frac{(-aq^2;q^2)_{n}z^nq^{n(n+1)}}{(q;q)_{2n+1}}=\frac{(-zq^2;q^2)_{\infty}}{(q;q^2)_{\infty}}\sum_{n=0}^{\infty}\frac{(azq;q^2)_n(-1)^nq^{n^2+2n}}{(q^2;q^2)_n(-zq^2;q^2)_n}.\label{eq:azeq28}
\end{equation}
\end{theorem}

\begin{proof}
This follows from the first line in Lemma \ref{lem1} with $q\to q^2$, $A=-a$, $B=q^{3}$, $x=-z$, and multiplying both sides by $(1-q)^{-1}$.
\end{proof}

\begin{cor}\label{cors1}
We have
\begin{equation}
\sum_{n\ge 0}\dfrac{(-z^{-1}q;q^2)_n z^n q^{n(n+1)}}{(q;q)_{2n+1}}=\dfrac{(-zq^2;q^2)_\infty}{(q;q^2)_\infty},    
\end{equation}
\begin{equation}
\sum_{n\ge 0}\dfrac{(-z^{-1}q^3;q^2)_n z^n q^{n(n+1)}}{(q;q)_{2n+1}}=\dfrac{(-zq^2;q^2)_\infty}{(q;q^2)_\infty}\sum_{n\ge 0}\dfrac{(-1)^n q^{n^2+2n}}{(-zq^2;q^2)_n},     
\end{equation}
\begin{equation}
\sum_{n\ge 0}\dfrac{q^{2n^2+n}}{(q;q)_{2n+1}}=(-q;q)_\infty,    
\end{equation}
\begin{equation}
\sum_{n\ge 0}\dfrac{q^{2n^2+3n}}{(q;q)_{2n+1}}=\dfrac{1}{(q;q^2)_\infty}\sum_{n\ge 0}(-1)^n q^{n^2+2n}.    
\end{equation}
\end{cor}
\begin{proof}
The first identity follows from \eqref{eq:azeq28} by choosing $a=(zq)^{-1}$; the second follows from \eqref{eq:azeq28} by choosing $a=z^{-1}q$. Finally, the third and fourth identities follow by letting $z\to 0$ in the first two identities, respectively.
\end{proof}
Now, setting $a = z = 1$ in \eqref{eq:azeq28}, we get
\begin{align*}
\sum\limits_{n=0}^{\infty}\frac{(-q^2;q^2)_nq^{n^2+n}}{(q;q)_{2n+1}} = \frac{(-q^2;q^2)_{\infty}}{(q;q^2)_{\infty}}\sum\limits_{n=0}^{\infty}\frac{(-1)^n(q;q^2)_nq^{n^2+2n}}{(q^4;q^4)_n}.\numberthis\label{eq:subseq28}
\end{align*}
Letting $x = -q^2$ and $y = q^2$ in the $\mathcal P_4$ refinement from \cite{DGLSIPI}, recorded in \eqref{eq:P4eq1}, and using \eqref{eq:subseq28}, we get
\begin{align*}
\sum\limits_{n=0}^{\infty}\frac{(-q^2;q^2)_nq^{n^2+n}}{(q;q)_{2n+1}} &= \frac{(-q^2;q^2)_{\infty}(q^3;q^2)_{\infty}}{(q;q^2)_{\infty}(q^4;q^2)_{\infty}}\left(1+\frac{(1-q)}{(1-q^2)}\sum\limits_{n=1}^{\infty}(1+q^{2n+1})q^{3n^2+2n}\right)\\
&= \frac{(-q^2;q^2)_{\infty}}{(q^2;q^2)_{\infty}}\sum\limits_{n=-\infty}^{\infty}q^{3n^2+2n}\\
&= \frac{(-q^2;q^2)_{\infty}(-q,-q^5,q^6;q^6)_{\infty}}{(q^2;q^2)_{\infty}}\numberthis\label{eq:JTPeq28},
\end{align*}
where \eqref{eq:JTPeq28} follows by replacing $q\rightarrow q^3$ and $z = q^2$ in Jacobi's Triple Product Identity. This proves ($28$) in Slater's list.

\begin{remark}
Setting $z=1$ in the first identity in Corollary \ref{cors1} gives us a ``companion" of Slater (28). Namely,
\begin{align}\label{slater28comp}
\sum_{n\ge 0}\dfrac{(-q;q^2)_nq^{n^2+n}}{(q;q)_{2n+1}}=\dfrac{(-q^2;q^2)_\infty}{(q;q^2)_\infty}=\dfrac{(q^4;q^4)_\infty}{(q;q)_\infty}.
\end{align}
\end{remark}

\begin{cor}\label{cornew}
We have
\begin{equation}
\begin{split}
 \sum_{n\ge 0}\dfrac{(-aq^2;q^2)_n z^nq^{n(n+1)}}{(q;q)_{2n}}&=\dfrac{(-zq^2;q^2)_\infty}{(q;q^2)_\infty}\sum_{n\ge 0} \dfrac{(azq^3;q^2)_n}{(q^2;q^2)_n(-zq^2;q^2)_n}(-1)^n q^{n^2}\\
&=
\frac{(-zq^2;q^2)_\infty}{(azq^4;q^2)_\infty}
\sum_{n\ge 0}
\frac{1-azq^{4n+2}}{1-azq^2}
\frac{(azq^2,-aq^2,azq^3;q^2)_n}
{(q^2,-zq^2,q;q^2)_n}
z^nq^{3n^2}.
\end{split}
\end{equation}
\end{cor}
\begin{proof}
Write
\begin{align*}
\sum_{n\ge 0}\dfrac{(-aq^2;q^2)_n z^nq^{n(n+1)}}{(q;q)_{2n}}&=\sum_{n\ge 0}\dfrac{(-aq^2;q^2)_n z^nq^{n(n+1)}}{(q^2;q^2)_{n}(q;q^2)_n}.
\end{align*}
Next, apply Lemma \ref{lem1} with $q\to q^2$, $A=-aq^2$, $B=q$, $x\to -zq^2$ to get
\begin{align*}
(q;q^2)_\infty\sum_{n\ge 0}\dfrac{(-aq^2;q^2)_n z^nq^{n(n+1)}}{(q^2;q^2)_{n}(q;q^2)_n}=(-zq^2;q^2)_\infty\sum_{n\ge 0} \dfrac{(azq^3;q^2)_n}{(q^2;q^2)_n(-zq^2;q^2)_n}(-1)^n q^{n^2}
\end{align*}
which yields the first equality. To get the second equality, we make the following specializations in \eqref{eq23}
\begin{equation*}
q\rightarrow q^2,\qquad
\alpha=azq^2,\qquad
\gamma=-aq^2,\qquad
\delta=azq^3,
\end{equation*}
and then let $\beta,\,\epsilon,\,N\to\infty.$
\end{proof}
\begin{remark}
Substituting $a=z=1/q$ in Corollary~\ref{cornew} recovers Slater's identity (29), which will be interpreted through the odd-position class in the next section.
\end{remark}
\begin{cor}\label{cor:S1-specializations}
The transformation in Corollary~\ref{cornew} yields the following useful specializations:
\begin{align}
\sum_{n\ge 0}\frac{(-z^{-1}q^{-1};q^2)_n z^n q^{n(n+1)}}{(q;q)_{2n}}
&=\frac{(-zq^2;q^2)_\infty}{(q;q^2)_\infty},\label{eq:S1-product-specialization}\\
\sum_{n\ge 0}\frac{z^nq^{n(n+1)}}{(q;q)_{2n}}
&=\frac{(-zq^2;q^2)_\infty}{(q;q^2)_\infty}
\sum_{n\ge 0}\frac{(-1)^nq^{n^2}}{(q^2;q^2)_n(-zq^2;q^2)_n},\label{eq:S1-a0-specialization}\\
\sum_{n\ge 0}\frac{(-z^{-1}q;q^2)_n z^nq^{n(n+1)}}{(q;q)_{2n}}
&=\frac{(-zq^2;q^2)_\infty}{(q;q^2)_\infty}
\sum_{n\ge 0}\frac{(-1)^nq^{n^2}}{(-zq^2;q^2)_n},\label{eq:S1-az-specialization}\\
\sum_{n\ge 0}\frac{q^{2n^2+n}}{(q;q)_{2n}}
&=\frac{1}{(q;q^2)_\infty}\sum_{n\ge 0}(-1)^nq^{n^2},\label{eq:S1-limit-specialization}\\
\sum_{n\ge 0}\frac{(-q;q^2)_nq^{n(n+1)}}{(q;q)_{2n}}
&=\frac{(-q^2;q^2)_\infty}{(q;q^2)_\infty}\,\phi(-q),\label{eq:S1-phi-specialization}\\
\sum_{n\ge 0}\frac{(-q^2;q^2)_nq^{n^2}}{(q;q)_{2n}}
&=\frac{(-q;q^2)_\infty}{(q;q^2)_\infty}\bigl(1+\psi(-q)\bigr).\label{eq:S1-psi-specialization}
\end{align}
Here $\phi(q)$ and $\psi(q)$ are Ramanujan's third-order mock theta functions.
\end{cor}
\begin{proof}
Choose $a=(zq^3)^{-1}$ in Corollary~\ref{cornew} to obtain the first identity. Next, setting $a=0$ and $az=q^{-1}$, respectively, gives the second and third identities. Letting $z\to0$ in the third identity yields the fourth identity. Setting $z=1$ in the third identity gives the fifth identity. Finally, taking $(a,z)=(1,q^{-1})$ in Corollary~\ref{cornew} gives the last identity.
\end{proof}

Letting $z = q$ and $a = q^{-1}$ in \eqref{eq:azS1}, we get
\begin{align*}
\sum\limits_{\lambda\in\mathcal{S}_1}q^{|\lambda|+\lfloor\frac{\#(\lambda)}{2}\rfloor-\overline{\#}(\lambda)} &= \sum\limits_{n=0}^{\infty}\frac{(-q;q^2)_nq^{n^2+2n}}{(q;q)_{2n+1}}\\
&= \frac{(-q^3;q^2)_{\infty}}{(q;q^2)_{\infty}}\sum\limits_{n=0}^{\infty}\frac{(-1)^n(q;q^2)_nq^{n^2+2n}}{(q^2;q^2)_n(-q^3;q^2)_n},\numberthis\label{eq:50rewrite}
\end{align*}
where \eqref{eq:50rewrite} follows from substituting $z = q$ and $a = q^{-1}$ in \eqref{eq:azeq28}. Substituting $(q,\alpha,\beta,\delta)\rightarrow (q^2,q^2,-q,q)$ in \eqref{eq:mastersubs}, from \eqref{eq:50rewrite}, we get
\begin{align*}
\sum\limits_{n=0}^{\infty}\frac{(-q;q^2)_nq^{n^2+2n}}{(q;q)_{2n+1}} &= \frac{(-q^3,q^3;q^2)_{\infty}}{(q,q^4;q^2)_{\infty}}\sum\limits_{n=0}^{\infty}q^{3n^2+3n}\\
&= \frac{(q^6;q^4)_{\infty}(q^{12};q^{12})_{\infty}}{(q,q^4;q^2)_{\infty}(q^6;q^{12})_{\infty}},
\end{align*}
where the last line above follows from Gauss' triangular number identity; see, for example, \cite{A}. Thus, we have
\begin{align*}
\sum\limits_{n=0}^{\infty}\frac{(-q;q^2)_nq^{n^2+2n}}{(q;q)_{2n+1}} = \frac{(q^2,q^{10},q^{12};q^{12})_{\infty}}{(q;q)_{\infty}}.
\end{align*}
This proves (50).

Now, letting $a = q^{-1}$ and $z = 1$ in \eqref{eq:azS1}, we get
\begin{align*}
\sum\limits_{\lambda\in\mathcal{S}_1}q^{|\lambda|-\overline{\#}(\lambda)} &= \sum\limits_{n=0}^{\infty}\frac{(-q;q^2)_nq^{n^2+n}}{(q;q)_{2n+1}}\\
&= \frac{(-q^2;q^2)_{\infty}}{(q;q^2)_{\infty}},\numberthis\label{eq:51rewrite}
\end{align*}
where \eqref{eq:51rewrite} follows from substituting $z = 1$ and $a = q^{-1}$ in \eqref{eq:azeq28}. Thus, we have
\begin{align*}
\sum\limits_{n=0}^{\infty}\frac{(-q;q^2)_nq^{n^2+n}}{(q;q)_{2n+1}} = \frac{(q^4,q^8,q^{12};q^{12})_{\infty}}{(q;q)_{\infty}}.
\end{align*}
This proves (51).

\section{Odd-position overpartitions and identities (29), (47), and (48)}\label{sec:S2}

In this section, we treat Equation $(29)$, $(47)$, and $(48)$ in Slater's list. They are given by
\begin{equation}\label{eq:Slater29}
\sum_{n=0}^{\infty}\frac{(-q;q^2)_{n}q^{n^2}}{(q;q)_{2n}}
=\dfrac{(-q;q^2)_{\infty}(-q^2,-q^4,q^6;q^6)_{\infty}}{(q^2;q^2)_{\infty}},
\end{equation}
\begin{equation}\label{eq:Slater47}
\sum_{n=0}^{\infty}\frac{(-1;q^2)_nq^{n^2}}{(q;q)_{2n}}
=\frac{(q^5,q^7,q^{12};q^{12})_{\infty}+q(q,q^{11},q^{12};q^{12})_{\infty}}{(q;q)_{\infty}},
\end{equation}
and
\begin{equation}\label{eq:Slater48}
\sum_{n=0}^{\infty}\frac{(-1;q^2)_nq^{n(n+1)}}{(q;q)_{2n}}
=\frac{(q^5,q^7,q^{12};q^{12})_{\infty}-q(q,q^{11},q^{12};q^{12})_{\infty}}{(q;q)_{\infty}}.
\end{equation}

Let $\mathcal{S}_2$ be the set of overpartitions $\lambda=(\lambda_1,\lambda_2,\ldots,\lambda_\ell)$ such that
\begin{enumerate}
    \item Only odd-indexed parts may be overlined.
    \item $\lambda_{2i-1}-\lambda_{2i}\geq1$, and $\lambda_{2i-1}-\lambda_{2i}\geq2$ if $\lambda_{2i-1}$ is overlined.
\end{enumerate}

The class $\mathcal{S}_2$ is likewise an SIP class of modulus $1$, and the basis $\mathcal{B}_{\mathcal{S}_2}$ is the set of partitions in $\mathcal{S}_2$ such that
\begin{enumerate}
    \item $\lambda_{2i}=\lambda_{2i+1}$ if $\lambda_{2i+1}>0$.
    \item $\lambda_{2i-1}-\lambda_{2i}=2$ if $\lambda_{2i-1}$ is overlined, and $\lambda_{2i-1}-\lambda_{2i}=1$ otherwise.
\end{enumerate}
Let $B_{\mathcal S_2}(n,h):=B_{\mathcal S_2}(n,h;a,z,q)$ be the generating function for partitions in $\mathcal{B}_{\mathcal{S}_2}$ with length $n$ and largest part $h$, where $a$, $z$, and $q$ keep track of the number of overlined parts, the number of odd-indexed parts, and the weight, respectively. We will show that
$$\sum_{\lambda\in\mathcal{S}_2}a^{\overline{\#}(\lambda)}z^{\lceil\frac{\#(\lambda)}{2}\rceil}q^{|\lambda|}=\sum_{n=0}^{\infty}\frac{(-aq;q^2)_{n}z^nq^{n^2}}{(q;q)_{2n}}.$$
Let $(a,z)\to(1,1)$, $(a,z)\to(q^{-1},1)$, and $(a,z)\to(q^{-1},q)$, we will have the sum sides of \eqref{eq:Slater29}, \eqref{eq:Slater47}, and \eqref{eq:Slater48}, respectively.

\begin{theorem}[Initial values for the $\mathcal S_2$ basis]\label{thm:S2-initial}
The initial values for $B_{\mathcal S_2}(n,h)$ are
\begin{equation*}
    B_{\mathcal S_2}(0,h)=\begin{cases}
        1 &  \text{if $h=0$,}\\
        0 &  \text{otherwise.}
    \end{cases}
\end{equation*}
\begin{equation*}
    B_{\mathcal S_2}(1,h)=\begin{cases}
        zq &  \text{if $h=1$,}\\
        azq^2 & \text{if $h=2$,}\\
        0 &  \text{otherwise.}
            \end{cases}
\end{equation*}
\end{theorem}

\begin{theorem}[Basis recurrence for $\mathcal S_2$]\label{thm:S2-recurrence}
For any $n,h\geq1$, $B_{\mathcal S_2}(n,h)$ satisfies the following recurrence.
\begin{equation}\label{eq:S2Rec}
B_{\mathcal S_2}(n,h)=zq^{2h-1}B_{\mathcal S_2}(n-2,h-1)+azq^{2h-2}B_{\mathcal S_2}(n-2,h-2).
\end{equation}
\end{theorem}

\begin{proof}
Given a partition $\lambda$ in $\mathcal{B}_{\mathcal{S}_2}$ with length $n$ and largest part $h$, then this largest part can be $h$ or $\overline{h}$. If the largest part is not overlined, the second and the third largest part of $\lambda$ will be $h-1$, and the weight of the first two parts is $zq^{2h-1}$. If the largest part is overlined, then the second and the third largest part of $\lambda$ will both be $h-2$, and the weight of the first two parts is $azq^{2h-2}$. So, by deleting the first two parts from $\lambda$, we have the desired recurrence.
\end{proof}

\begin{theorem}[Closed basis forms for $\mathcal S_2$]\label{thm:S2-closed}
The closed forms for $B_{\mathcal S_2}(n,h)$ are given by
\begin{equation}\label{eq:S2B1}
B_{\mathcal S_2}(2n,n+h+1)=a^hz^nq^{n^2+2n+h^2}{n\brack h}_{q^2}
\end{equation}
and
\begin{equation}\label{eq:S2B2}
B_{\mathcal S_2}(2n-1,n+h)=a^hz^nq^{n^2+h^2}{n\brack h}_{q^2}.
\end{equation}
\end{theorem}

\begin{proof}
By letting $n\to0$ in \eqref{eq:S2B1} and $n\to1$ in \eqref{eq:S2B2}, we can easily check the initial values. So, it remains to check the recurrence relation. For \eqref{eq:S2B1},
\begin{align*}
B_{\mathcal S_2}(2n,n+h+1)=&a^hz^nq^{n^2+2n+h^2}{n\brack h}_{q^2}\\
=&a^hz^nq^{n^2+2n+h^2}\left(q^{2h}{n-1\brack h}_{q^2}+{n-1\brack{h-1}}_{q^2}\right)\\
=&a^hz^nq^{n^2+2n+h^2+2h}{n-1\brack h}_{q^2}+a^hz^nq^{n^2+2n+h^2}{n-1\brack{h-1}}_{q^2}\\
=&zq^{2n+2h+1}\cdot a^{h}z^{n-1}q^{(n-1)^2+2(n-1)+h^2}{n-1\brack h}_{q^2}\\
&+azq^{2n+2h}\cdot a^{h-1}z^{n-1}q^{(n-1)^2+2(n-1)+(h-1)^2}{n-1\brack{h-1}}_{q^2}\\
=&zq^{2n+2h+1}B_{\mathcal S_2}(2n-2,n+h)+azq^{2n+2h}B_{\mathcal S_2}(2n-2,n+h+1).
\end{align*}
Similarly, for \eqref{eq:S2B2},
\begin{align*}
B_{\mathcal S_2}(2n-1,n+h)=&a^hz^nq^{n^2+h^2}{n\brack h}_{q^2}\\
=&a^hz^nq^{n^2+h^2}\left(q^{2h}{n-1\brack h}_{q^2}+{n-1\brack h-1}_{q^2}\right)\\
=&a^hz^nq^{n^2+h^2+2h}{n-1\brack h}_{q^2}+a^hz^nq^{n^2+h^2}{n-1\brack h-1}_{q^2}\\
=&zq^{2n+2h-1}\cdot a^hz^{n-1}q^{(n-1)^2+h^2}{n-1\brack h}_{q^2}\\
&+azq^{2h+2h-2}\cdot a^{h-1}z^{n-1}q^{(n-1)^2+(h-1)^2}{n-1\brack h-1}_{q^2}\\
=&zq^{2n+2h-1}B_{\mathcal S_2}(2n-3,n+h-1)+azq^{2n+2h-2}B_{\mathcal S_2}(2n-3,n+h-2).
\end{align*}
This completes the proof.
\end{proof}

\begin{theorem}\label{thm:S2-gf}
The following three-variable generating function holds for $\mathcal{S}_2$.
\begin{equation}\label{eq:azS2}
\sum_{\lambda\in\mathcal{S}_2}a^{\overline{\#}(\lambda)}z^{\lceil\frac{\#(\lambda)}{2}\rceil}q^{|\lambda|}=\sum_{n=0}^{\infty}\frac{(-aq;q^2)_{n}z^nq^{n^2}}{(q;q)_{2n}}.
\end{equation}
\end{theorem}

\begin{proof}
Since $\mathcal{S}_2$ is an SIP class of modulus $1$,
\begin{align*}
\sum_{\lambda\in\mathcal{S}_2}a^{\overline{\#}(\lambda)}z^{\lceil\frac{\#(\lambda)}{2}\rceil}q^{|\lambda|}=&\sum_{n=0}^{\infty}\sum_{h=0}^{\infty}\frac{B_{\mathcal S_2}(n,h)}{(q;q)_n}\\
=&\sum_{n=0}^{\infty}\sum_{h=0}^{\infty}\frac{B_{\mathcal S_2}(2n,n+h+1)}{(q;q)_{2n}}+\sum_{n=1}^{\infty}\sum_{h=0}^{\infty}\frac{B_{\mathcal S_2}(2n-1,n+h)}{(q;q)_{2n-1}}\\
=&\sum_{n=0}^{\infty}\frac{z^nq^{n^2+2n}}{(q;q)_{2n}}\sum_{h=0}^{\infty}a^hq^{h^2}{n\brack h}_{q^2}+\sum_{n=1}^{\infty}\frac{z^nq^{n^2}}{(q;q)_{2n-1}}\sum_{h=0}^{\infty}a^hq^{h^2}{n\brack h}_{q^2}\\
=&\sum_{n=0}^{\infty}\frac{(-aq;q^2)_{n}z^nq^{n^2+2n}}{(q;q)_{2n}}+\sum_{n=1}^{\infty}\frac{(-aq;q^2)_{n}z^nq^{n^2}}{(q;q)_{2n-1}}\\
=&1+\sum_{n=1}^{\infty}\frac{(-aq;q^2)_{n}z^nq^{n^2}}{(q;q)_{2n-1}}\left(\frac{q^{2n}}{1-q^{2n}}+1\right)\\
=&\sum_{n=0}^{\infty}\frac{(-aq;q^2)_{n}z^nq^{n^2}}{(q;q)_{2n}}.
\end{align*}
This completes the proof.
\end{proof}

\begin{remark}\label{rem:S1e-S2}
At the generating-function level, the $\mathcal S_2$ family is obtained from the even-length part of the $\mathcal S_1$ family by a simple change of variables. More precisely, the first identity in Corollary~\ref{cor:S1-evenodd} gives
\[
\sum_{\lambda\in\mathcal S_2}a^{\overline{\#}(\lambda)}z^{\lceil\#(\lambda)/2\rceil}q^{|\lambda|}
=
\sum_{\mu\in\mathcal S_{1,e}}
\left(\frac{a}{q}\right)^{\overline{\#}(\mu)}
\left(\frac{z}{q}\right)^{\#(\mu)/2}q^{|\mu|}.
\]
Indeed, this substitution changes the right-hand side of Corollary~\ref{cor:S1-evenodd} into the series in \eqref{eq:azS2}. This relation explains why several analytic consequences of the two families are closely connected.
\end{remark}

Note that, just as in the proof of Theorem~\ref{thm:S1}, we may divide the partitions by the parity of their length. Thus, we have the following generating functions.

\begin{cor}
Let $\mathcal{S}_{2,e}$ and $\mathcal{S}_{2,o}$ be the sets of partitions in $\mathcal{S}_2$ with length being even and odd, respectively. Then,
\begin{equation}
\sum_{\lambda\in\mathcal{S}_{2,e}}a^{\overline{\#}(\lambda)}z^{\lceil\frac{\#(\lambda)}{2}\rceil}q^{|\lambda|}=\sum_{n=0}^{\infty}\frac{(-aq;q^2)_{n}z^nq^{n^2+2n}}{(q;q)_{2n}}
\end{equation}
and
\begin{equation}
\sum_{\lambda\in\mathcal{S}_{2,o}}a^{\overline{\#}(\lambda)}z^{\lceil\frac{\#(\lambda)}{2}\rceil}q^{|\lambda|}=\sum_{n=1}^{\infty}\frac{(-aq;q^2)_{n}z^nq^{n^2}}{(q;q)_{2n-1}}=\sum_{n=0}^{\infty}\frac{(-aq;q^2)_{n+1}z^{n+1}q^{(n+1)^2}}{(q;q)_{2n+1}}.
\end{equation}
\end{cor}
We now evaluate the two remaining Slater series associated with $\mathcal S_2$.

Letting $a = q^{-1}$ and $z = 1$ in \eqref{eq:azS2}, we get
\begin{align*}
\sum\limits_{\lambda\in\mathcal{S}_2}q^{|\lambda|-\overline{\#}(\lambda)} &= \sum\limits_{n=0}^{\infty}\frac{(-1;q^2)_nq^{n^2}}{(q;q)_{2n}}\\
&= \sum\limits_{n=0}^{\infty}\frac{(-1;q^2)_nq^{n^2}}{(q^2;q^2)_{n}(q;q^2)_{n}}.\numberthis\label{eq:azS2rewrite47}
\end{align*}
Substituting $(q,\alpha,\beta,\delta)\rightarrow (q^2,q^{-1},1,-1)$ in \eqref{eq:mastersubs}, from \eqref{eq:azS2rewrite47}, we get
\begin{align*}
\sum\limits_{n=0}^{\infty}\frac{(-1;q^2)_nq^{n^2}}{(q^2;q^2)_{n}(q;q^2)_{n}} &= \frac{(-q;q^2)_{\infty}}{(q;q^2)_{\infty}}\\
&= \frac{(-q,q^2;q^2)_{\infty}}{(q;q)_{\infty}}\\
&= \frac{(-q;-q)_{\infty}}{(q;q)_{\infty}}.\numberthis\label{eq:negEPNT}
\end{align*}
We note that \eqref{eq:negEPNT} also follows from \eqref{eq:S1-product-specialization} with $z=q^{-1}$. Using Euler's pentagonal number theorem \cite{A} with $q\rightarrow -q$ in \eqref{eq:negEPNT}, we have
\begin{align*}
\sum\limits_{n=0}^{\infty}\frac{(-1;q^2)_nq^{n^2}}{(q^2;q^2)_{n}(q;q^2)_{n}} &= \frac{1}{(q;q)_{\infty}}\sum\limits_{n=-\infty}^{\infty}(-q)^{\frac{n(3n-1)}{2}}\\
&= \frac{1}{(q;q)_{\infty}}\left(\sum\limits_{n=-\infty}^{\infty}(-1)^nq^{6n^2-n} + \sum\limits_{n=-\infty}^{\infty}(-1)^nq^{6n^2+5n+1}\right)\\
&= \frac{(q^5,q^7,q^{12};q^{12})_{\infty}+q(q,q^{11},q^{12};q^{12})_{\infty}}{(q;q)_{\infty}},
\end{align*}
where the last line above follows using Jacobi's Triple Product Identity which gives us (47).

Now, letting $a = q^{-1}$ and $z = q$ in \eqref{eq:azS2}, we get
\begin{align*}
\sum\limits_{\lambda\in\mathcal{S}_2}q^{|\lambda|+\lceil\frac{\#(\lambda)}{2}\rceil-\overline{\#}(\lambda)} &= \sum\limits_{n=0}^{\infty}\frac{(-1;q^2)_nq^{n^2+n}}{(q;q)_{2n}}\\
&= \sum\limits_{n=0}^{\infty}\frac{(-1;q^2)_nq^{n^2+n}}{(q^2;q^2)_{n}(q;q^2)_{n}}.\numberthis\label{eq:azS2rewrite48}
\end{align*}
Substituting $(q,\alpha,\beta,\delta)\rightarrow (q^2,1,q,-1)$ in \eqref{eq:mastersubs}, from \eqref{eq:azS2rewrite48}, we get
\begin{align*}
\sum\limits_{n=0}^{\infty}\frac{(-1;q^2)_nq^{n^2+n}}{(q^2;q^2)_{n}(q;q^2)_{n}} &= \frac{(-q^2;q^2)_{\infty}}{(q^2;q^2)_{\infty}}\left(1+2\sum\limits_{n=1}^{\infty}q^{3n^2}\right)\\
&= \frac{(q,-q^2;q^2)_{\infty}}{(q;q)_{\infty}}\sum\limits_{n=-\infty}^{\infty}q^{3n^2}\\
&= \frac{(q,-q^2;q^2)_{\infty}(-q^3,-q^3,q^6;q^6)_{\infty}}{(q;q)_{\infty}},
\end{align*}
where the last line above follows using Jacobi's Triple Product Identity. We then have
\begin{align*}
\sum\limits_{n=0}^{\infty}\frac{(-1;q^2)_nq^{n^2+n}}{(q^2;q^2)_{n}(q;q^2)_{n}} &= \frac{(q,q^3,q^5,-q^2,-q^4,-q^6,-q^3,-q^3,q^6;q^6)_{\infty}}{(q;q)_{\infty}}\\
&= \frac{(q,-q^2,-q^3,-q^4,q^5,q^6;q^6)_{\infty}}{(q;q)_{\infty}}\\
&= \frac{1}{(q;q)_{\infty}}\sum\limits_{n=-\infty}^{\infty}(-1)^nq^{6n^2-n}(1-q^{6n+1}),
\end{align*}
where the last line above follows using Quintuple Product Identity. Thus, we have
\begin{align*}
\sum\limits_{n=0}^{\infty}\frac{(-1;q^2)_nq^{n^2+n}}{(q^2;q^2)_{n}(q;q^2)_{n}} &= \frac{1}{(q;q)_{\infty}}\left(\sum\limits_{n=-\infty}^{\infty}(-1)^nq^{6n^2-n} - \sum\limits_{n=-\infty}^{\infty}(-1)^nq^{6n^2+5n+1}\right)\\
&= \frac{(q^5,q^7,q^{12};q^{12})_{\infty}-q(q,q^{11},q^{12};q^{12})_{\infty}}{(q;q)_{\infty}},
\end{align*}
where the last line above follows from Jacobi's triple product identity, yielding (48).

\section{SIP and other constructive combinatorics}
\label{sec:applications}

The constructions developed in this paper naturally belong to a broader combinatorial landscape that includes the classical insertion methods of Joichi and Stanton \cite{JoichiStanton1987} and their subsequent use in the theory of overpartitions. Their constructive combinatorial ideas have been utilized, for example, in \cite[Proposition~2.1]{overpartitions}. The proposition gives a direct construction of overpartitions from an ordinary partition together with a partition into distinct bounded parts: each selected part determines a prefix addition, followed by the creation of an overline at the endpoint of that prefix. Thus, several of the sum-side constructions in the present paper may also be described through variants of a Joichi--Stanton type insertion.

This relation should be viewed as an additional structural feature of the present results, rather than as a competing interpretation. Indeed, the two viewpoints operate at different levels. A Joichi--Stanton type map typically begins with a fixed partition and records a finite collection of independent local choices. These choices explain the appearance of a finite $q$-Pochhammer factor. By contrast, an SIP description begins by identifying a canonical family of minimal configurations and then proves that every partition in the class has a unique decomposition into one of these configurations together with a freely extendable tail.

This distinction is especially useful in the present work. The strict overpartition class of Section~3 has a natural staircase-type basis, and its refinement yields the series associated with Slater's identity (12). For the overpartition classes in Sections~4 and~5, the admissible insertions depend on the parity of the position of a part.  The fact that overlining is allowed only at alternating positions is therefore not merely a marking convention: it is part of the structure of the basis itself. The SIP viewpoint keeps track simultaneously of the position, the allowed gap, the overlining statistic, and the freely extendable tail. This produces the multivariate generating functions from which the specializations corresponding to Slater's identities~(28), (29), (47), (48), (50), and (51) arise.

The usefulness of this organization is also visible in the companion identities.  Once a basis has been identified, modifications such as a shift in the smallest-part condition or a uniform displacement of the basis are visible directly at the partition-theoretic level. In the shifted mod-$4$ family of Section~3, this leads to the new G\"ollnitz--Gordon type companion and its signed analogue. Likewise, the even-position and odd-position classes lead to companion transformations that are not apparent from the original entries of Slater's list. Thus the role of SIP in this paper is not only to verify that a given series has a partition interpretation. It provides a framework in which refinements, specializations, shifts of the underlying basis, and new companion identities can be generated from the same structural source.

Accordingly, we do not claim that an SIP interpretation excludes the existence of a classical constructive proof. A single Rogers--Ramanujan type identity may possess several proofs, several bijections, and several natural partition-theoretic meanings.  As George E.~Andrews emphasized to us in a correspondence, this is an ``\emph{embarrassment of riches}'', not a defect of the theory. The value of the SIP framework is its uniformity: it places different identities and different partition families within a single basis--tail formalism, while retaining enough information to produce multivariate refinements and to suggest new identities.

This perspective is consistent with the broader program advocated by Andrews in his work on linked partition ideals \cite{AndrewsLinked}. A useful theory need not settle every problem by one universal construction. Rather, its vitality comes from the coexistence of broad methods that organize many examples and difficult cases that reveal genuinely unexplored structure. In this sense, the interaction between SIP classes and Joichi--Stanton type insertions points toward two complementary directions. One is to determine when a classical insertion construction realizes the basis of an SIP class. The other is to identify SIP classes whose canonical bases lead to refinements or companion identities that are not naturally suggested by a pre-existing insertion. The new companions obtained here provide initial evidence for the second direction. A further indication of the scope of the SIP method comes from the Rogers--Selberg identities. These identities have a substantially more rigid local structure than the examples considered above, and, to the best of our knowledge, they do not presently admit a comparably direct classical constructive treatment through the usual prefix-insertion mechanisms. Our ongoing work shows that their gap conditions nevertheless fit naturally into the SIP framework. The associated canonical bases encode the relevant local restrictions in a systematic way, while the SIP decomposition exposes shifts and refinements of those bases that lead to several Rogers--Selberg companion identities. Thus, even when a classical insertion interpretation is not immediately accessible, SIP analysis can remain effective: it identifies the correct minimal objects, organizes the allowed tails, and turns structural changes in the basis into new $q$-series identities.  This suggests that the value of SIP is not limited to providing another explanation for known constructions, but also lies in making difficult families amenable to a coherent combinatorial analysis.

\section*{Acknowledgments}
We thank Jeremy Lovejoy for bringing \cite{overpartitions} and \cite{JoichiStanton1987} to our attention and for prompting the comparison developed in the previous section. We also thank George E.~Andrews for a valuable private communication concerning the role of uniform structural methods in the study of Rogers--Ramanujan type identities.

\section*{}
\subsection*{Data availability}
No data was used for the research described in the article.

\end{document}